\documentclass[12pt]{article}

\usepackage{amssymb}
\usepackage{amsthm}
\usepackage{latexsym}
\usepackage{amsmath}

\def\dfrac{\displaystyle\frac}
\def\dint{\displaystyle\int}

\def\dsum{\displaystyle\sum}
\def\dlim{\displaystyle\lim}
\def\esssup{\mathrm{ess}\sup}

\def\F{{\mathcal F}}

\def\D{\displaystyle}

\newtheorem{theorem}{Theorem}

\newtheorem{lemma}{Lemma}
\newtheorem{proposition}{Proposition}
\newtheorem{remark}{Remark}

\begin{document}

\def\second {\vtop{\baselineskip=11pt
\hbox to 125truept{\hss Shigeyoshi OGAWA\thanks{Dept. of Math. Science, 
Ritsumeikan University, ogawa-s@se.ritsumei.ac.jp}\hss}
}}
\def\third {\vtop{\baselineskip=11pt
\hbox to 125truept{\hss Hideaki UEMURA\thanks{Dept. of Math. Education,
Aichi University of Education, huemura@auecc.aichi-edu.ac.jp}\hss}
}}
\title{{Identification of noncausal It\^{o} processes from the stochastic Fourier 
coefficients II
}}

\author{\second\and\third}
\date{}

\maketitle
%
%

\section{Introduction} 
Let $a(t,\omega)$ be a random function on $[0,1]\times\Omega$, $(\Omega, \F,P)$ being a probability space.
Let  $\{\varphi_n(t)\}$ be an orthonormal basis in $L^2([0,1])$. 
We consider the following stochastic integral;
\begin{equation}\label{eq:SFC}
\mathcal{F}_n(a)=\int_0^1 \overline{\varphi_n(t)}a(t,\omega)dW_t, 
\end{equation}
where  $W.$ is a standard real Brownian motion and $\overline{\varphi_n(t)}$ denotes the complex conjugate of  $\varphi_n(t)$.
(We will soon give several kinds of conditions on $a(t,\omega)$ and the meaning of the stochastic integral.)
We call $\mathcal{F}_n(a)$ the stochastic Fourier coefficient (SFC in abbr.) of $a(t,\omega)$.
Our problem is "to
reconstruct $a(t,\omega)$ from the complete set of SFCs $\{\mathcal{F}_n(a)\}$".

\

In the preceding articles \cite{ogawa, ogawa-uemura}, we have studied this 
problem in such special case  
that the function $a(t,\omega)$ 
is restricted to the square integrable Wiener functional, namely; \ \  
We take a Wiener space for $(\Omega,\F, P)$ and denote by ${\mathcal L}^2$ the subclass 
$L^2([0,1]\times \Omega, dt\times dP)$ of all square integrable functions. 
We call the element $a(t,\omega)$ of ${\mathcal L}^2$ a square integrable Wiener functional and in particular a causal Wiener functional provided that $a(t,\omega)$ is adapted to the natural filtration $\{\F_t\}_{t>0}$ associated to the Brownian motion. 
We assume that an orthonormal basis $\{\varphi_n\}$ in $L^2(0,1)$ satisfies 
\begin{itemize} 
\item[(A)] \qquad \qquad $\D{\sup_{t \in [0,1]}}|\varphi_n(t)| < \infty \quad \mbox{for all}\quad n$.
\end{itemize}
In this case the SFC \eqref{eq:SFC} is defined by It\^{o} integral.
\begin{theorem}[case of causal functions, \cite{ogawa}]\label{th-0}
For any orthonormal basis $\{\varphi_n\}$ in $L^2(0,1)$ satisfying Condition (A),
every causal Wiener functional $a(t,\omega)\in{\mathcal L}^2$ can be reconstructed from the complete set of SFCs of the $a(t,\omega)$.
\end{theorem}
After this result an extension to the case of noncausal Wiener functionals was 
established in \cite{ogawa-uemura} in the following way; 
Let ${\mathcal L}^{2,1} \subset{\mathcal L}^2$ be the subclass of all square integrable 
Wiener functionals for which the Skorokhod integral is well defined. 
We notice that for every $a(t,\omega)$ of the class 
${\mathcal L}^{2,1}$ its stochastic derivative $D_sa(t,\omega)$, introduced by A.Skorokhod 
\cite{skorokhod} is also well defined. 
We also assume that an orthonormal basis $\{\varphi_n\}$ in $L^2(0,1)$ satisfies Condition (A). 
For any element $a(t,\omega)$ of this subclass the SFC \eqref{eq:SFC} is defined by the Skorokhod integral
(See, for instance, Ogawa \& Uemura \cite{ogawa-uemura}, Nualart \cite{Nualart} for the definitions of ${\mathcal L}^{2,1}$, the Skorokhod integral and the stochastic derivative $D_sa(t,\omega)$).
Then we have the following results; 
\begin{theorem}[case of noncausal functions \cite{ogawa-uemura}]\label{th-1}~
\begin{enumerate}
\item \ \ Every $a(t,\omega)\in{\mathcal L}^{2,1}$ can be reconstructed from the set of SFCs $\{\F_n(a),$ $n \in{\bf Z}\}$ with respect to the complete system of trigonometric functions $e_n(t)=\exp(2\pi\sqrt{-1} nt), \ n \in {\bf Z}$.  
\item \ \ Moreover if the function $a(t,\omega) \in {\mathcal L}^{2,1}$ satisfies the 
condition 
\begin{equation*} \label{eq:th2}
\sup_{0 \leq t \leq 1}\left\{E[|a(t,\omega)|^2]+\int_0^1E[|D_sa(t,\omega)|^2]ds\right\} < \infty,
\end{equation*}
then the above result (i) holds for any orthonormal basis $\{\varphi_n\}$ satisfying Condition (A).
\end{enumerate}
\end{theorem}
We note that Theorems \ref{th-0} and \ref{th-1} are verified by the reconstruction of all kernel functions of Wiener chaos decomposition for $a(t,\omega)$ by use of appropriate multiple Wiener integrals.

\ 

Let $a(t,\omega)\in\mathcal{L}^{2,1}$ and $b(t,\omega)\in\mathcal{L}^{2}$ be noncausal functions.
We next consider a noncausal It\^{o} process of the form $dX_t=b(t,\omega)dt+a(t,\omega)dW_t$,
where $dW_t$ is defined by the Skorokhod integral.
We set the SFC $\mathcal{F}_n(dX)$ of the differential $dX$ by
\begin{equation}\label{eq:riemann integral}
\F_n(dX)=\int_0^1 \overline{\varphi_n(t)}dX_t.
\end{equation}
Here we understand by the stochastic integral on the right hand side of 
\eqref{eq:riemann integral} the integral defined as the limit of Riemann 
sums,
$$
\lim_{|\Delta| \to 0}\sum_{t_i \in \Delta}\overline{\varphi_n(t_i)}\{X(t_{i+1})-X(t_i)\}, 
$$
where \ $\Delta=\{0=t_0<t_1<\cdots<t_n<t_{n+1}=1\}$ is a partition of the interval 
$[0,1]$ and \ $|\Delta|=\max_i\{t_{i+1}-t_i\}$. 
It is easy to see that \eqref{eq:riemann integral} coincides with
\begin{equation*}
\int_0^1 \overline{\varphi_n(t)}b(t,\omega)dt+ \int_0^1 \overline{\varphi_n(t)}a(t,\omega)dW_t 
\end{equation*}
if $\varphi_n(t)\in C([0,1])$.

Given these we are now concerned with the basic question whether we can identify the 
two parameters $a(\cdot,\omega)$ and $b(\cdot,\omega)$
of the system from the complete set of SFCs 
$\{\F_n(dX)\}$. 
Unfortunately, our "multiple Wiener integrals method" in \cite{ogawa, ogawa-uemura} is of no use to apply for $dX$, 
since we cannot distinguish between $a(\cdot,\omega)$ and $b(\cdot,\omega)$ by this method.

In  \cite{ogawa-uemura2}, we have discussed the reconstruction of noncausal It\^{o} process from its stochastic Fourier coefficients with respect to the complete system of trigonometric functions by use of the method of the Bohr convolution.
On that occasion we assume some regularities on $a(t,\omega)$ and $b(t,\omega)$
in order that the theory of Skorokhod integrals is applicable to corresponding noncausal It\^{o} process;
\begin{theorem}[case of noncausal It\^{o} process \cite{ogawa-uemura2}] \label{th-2}
Suppose $a(t,\omega) \in \mathcal{L}^{2,2}$ and $b(t,\omega) \in \mathcal{L}^{2,1}$.
Then 
every noncausal It\^{o} process can be identified by the complete set of the SFCs 
$\{{\mathcal F}_n(dX),n \in {\bf Z}\}$ of the differential $dX$ with respect to the complete system 
of trigonometric functions $\{e_n(t), \ n \in {\bf Z}\}$ where $e_n(t)=e^{2\pi\sqrt{-1}nt}$. 
More precisely, the Fourier coefficient (in the usual sense) $\tilde{a}_n$ of the 
coefficient $a(t,\omega)$ is determined by the following formula of Bohr product;
\begin{equation*} 
\begin{array}{l}
\D{\lim_{N \rightarrow \infty}\frac{1}{2N+1}\sum_{k+\ell=n, |\ell|\leq N}
\!\!\left(\int _0^1\!\!a(t,\omega)\overline{e_k(t)} dW_t+ \int_0^1 \!\!b(t,\omega)\overline{e_k(t)} 
dt\right)\!\int_0^1\overline{e_{\ell}(t)}dW_t} \\
=\D{\int_0^1a(t,\omega)\overline{e_n(t)}dt} \quad \mathrm{in} \quad L^2(\Omega,dP).
\end{array} 
\end{equation*}
\end{theorem}
Here $\mathcal{L}^{2,2}$ is a subclass of $\mathcal{L}^{2,1}$.
See, for instance, Ogawa \& Uemura \cite{ogawa-uemura2}, Nualart \cite{Nualart} for the definitions of ${\mathcal L}^{2,2}$.
Since the set $\{\tilde{a}_n\}$ reconstructs $a(t,\omega)$, we obtain $\{\mathcal{F}_n(a)\}$ from $\{\tilde{a}_n\}$.
Thus we get $\{\tilde{b}_n\}$ from the equality $\tilde{b}_n=\mathcal{F}_n(dX)-\mathcal{F}_n(a)$, which reconstructs $b(t,\omega)$.

%
\

In this note, we study Theorem \ref{th-2} above under some mild conditions on $a(t,\omega)$ and $b(t,\omega)$. 
For example, we assume $E[(\int_0^1|a(t,\omega)|^2dt)^{p/2}]<\infty$ for some $p\in(1,2)$. 
Under this assumption, we cannot apply the Skorokhod integral to the integrand $a(t,\omega)$, since the theory of Skorokhod integrals deeply depends upon the Wiener It\^{o} decomposition of the integrand. 
Thus we will employ another stochastic calculus, that is, the Malliavin calculus.
In this framework, Skorokhod integral is extended to \textit{the divergence}, and we can apply it to the integrand $a(t,\omega)$ above or more general integrands.

The aim of the present note is to generalize Theorem \ref{th-2} above in the framework of Malliavin calculus.
In the next Section 2 we will sketch the preparations of Malliavin calculus only for the later use. 
Then we will give a statement and proof of our main result (Theorem \ref{thm:Main2}) in Section 3, after 
some modification to the definition of SFCs for further use.
Section 4 is devoted to 
the proof of the propositions by which the main theorem is verified. 
\section{Preliminaries}
In this section we first sketch the preparations of Malliavin calculus only for the later use. 
Refer, for instance, Ikeda and Watanabe \cite{IkedaWatanabe}, Shigekawa \cite{Shigekawa} for details.\par
Let $(\Omega, P)$ be the Wiener space, i.e.,
\begin{equation*}
\Omega = \{ W_t : [0,1] \to {\mathbf R} \,|\, W_t \text{~is~continuous~and~} W_0 = 0 \}
\end{equation*}
and $P$ is the standard Wiener measure. 
We denote the Cameron-Martin subspace of $\Omega$ by $H$.
Let $K$ be a separable Hilbert space and ${\mathbb D}_p^{\alpha}(K)\, ( 1< p <\infty, \alpha \in {\mathbb
R} )$  be Watanabe's Sobolev spaces of $K$-valued (generalized) Wiener functionals, i.e.,
 \begin{equation*}
  {\mathbb D}_p^{\alpha}(K) = ( I - L )^{-\alpha/2} L^p( \Omega\to K, dP )%
 \end{equation*}%
endowed with the norm $\|F(\omega)\|_{{\mathbb D}_p^{\alpha}(K)} = E[| ( I-L )^{\alpha/2} F(\omega) |_K^p]^{1/p}$,
where  $L$  denotes the Ornstein-Uhlenbeck operator and $|\cdot|_K$ the norm of $K$.
We note that ${\mathbb D}_p^0(K)= L^p( \Omega\to K, dP )$, ${\mathbb D}_2^1(L^2([0,1]))=\mathcal{L}^{2,1}$, ${\mathbb D}_2^2(L^2([0,1]))=\mathcal{L}^{2,2}$ and that ${\mathbb D}_p^{\alpha}(K)\subset{\mathbb D}_{p'}^{\alpha'}(K)$ holds if $p\ge p'$ and $\alpha\ge \alpha'$.
We also note that ${\mathbb D}_p^{\alpha}(K)$ is the completion of $\mathcal{P}(K)$ under the norm $\|\cdot\|_{{\mathbb D}_p^{\alpha}(K)}$, where $\mathcal{P}(K)$ denotes the set of $K$-valued polynomials on $\Omega$. 
If $q$ is the conjugate exponent of $p$, i.e. $1/p+1/q=1$, then ${\mathbb D}_q^{-\alpha}(K)$ is the dual space of ${\mathbb D}_p^{\alpha}(K)$.
If $F(\omega) \in {\mathbb D}_p^{\alpha}(K)$ and $G(\omega) \in {\mathbb D}_q^{-\alpha}(K)$, the coupling ${}_{{\mathbb D}_p^{\alpha}(K)}\langle F(\omega), G(\omega)\rangle_{{\mathbb D}_q^{-\alpha}(K)}$ of $F(\omega)$ and $G(\omega)$ is defined by the following;
\begin{equation*}
{}_{{\mathbb D}_p^{\alpha}(K)}\langle F(\omega), G(\omega)\rangle_{{\mathbb D}_q^{-\alpha}(K)}=E[\langle ( I-L )^{\alpha/2} F(\omega), ( I-L )^{-\alpha/2} G(\omega)\rangle_{K}],
\end{equation*}
 $\langle\ast,\star\rangle_{K}$ denoting the inner product of $K$.
Let $D : {\mathbb D}_p^{\alpha}(K) \to {\mathbb D}_p^{\alpha-1}(H\otimes K)$ be the $H$-differentiation, and \textit{the divergence} operator $D^* : {\mathbb D}_q^{-(\alpha-1)}(H\otimes K) \to {\mathbb D}_q^{-\alpha}(K)$ be the dual operator of $D$, where $q$ is the conjugate exponent of $p$. 
Here $H\otimes K$ denotes the tensor product of $H$ and $K$, i.e., the totality of linear mappings from $H$ to $K$ of Hilbert-Schmidt type.
If $h(t,\omega)\in\mathbb{D}_p^{\alpha}(L^2([0,1]\to K,dt))$, then 
$$\mathsf{h}(\omega)=(I-L)^{-\alpha/2}\dint_0^{\cdot}(I-L)^{\alpha/2}h(t,\omega)dt\in \mathbb{D}_p^{\alpha}(H\otimes K).$$
We note that $\|\mathsf{h}\|_{{\mathbb D}_p^{\alpha}(H\otimes K)}=\|h(t,\omega)\|_{{\mathbb D}_p^{\alpha}(L^2([0,1],dt)\otimes K)}$.
We denote $D^*\mathsf{h}(\in{\mathbb D}_p^{\alpha-1}(K))$ by $\int_0^1h(t,\omega)dW_t$.
The following estimates are satisfied for some positive constants $C_1$ and $C_2$:
\allowdisplaybreaks{
\begin{align}
&\left\|\dint_0^1h(t,\omega)dW_t\right\|_{{\mathbb D}_p^{\alpha-1}(K)}\le C_1 \|h(t,\omega)\|_{{\mathbb D}_p^{\alpha}(L^2([0,1],dt)\otimes K)},\label{estimate:dW}\\
&\|DF(\omega)\|_{{\mathbb D}_p^{\alpha-1}(H\otimes K)}\le C_2\|F(\omega)\|_{{\mathbb D}_p^{\alpha}(K)}.\label{estimate:D}
\end{align}
}
We also note that the $H$-differentiation $D$ accompanies the operator $D_{\cdot}:\mathbb{D}_p^{\alpha}(K)\to\mathbb{D}_p^{\alpha-1}(L^2([0,1],dt)\otimes K)$ in the following manner:
\begin{equation*}
DF(\omega)[h]=(I-L)^{-(\alpha-1)/2}\int_0^1(I-L)^{(\alpha-1)/2}D_tF(\omega)h'(t)dt,
\end{equation*}
$F(\omega)\in{\mathbb D}_p^{\alpha}(K)$ and $h=\int_0^{\cdot}h'(t)dt\in H$.%

\begin{remark}\label{remark:int_hdt}
If $h(t,\omega)\in\mathbb{D}_p^{\alpha}(L^2([0,1]\to K,dt))$, then $\int_0^1h(t,\omega)dt$ is defined by $(I-L)^{-\alpha/2}\int_0^1(I-L)^{\alpha/2}h(t,\omega)dt$.
Indeed, let $\{ h_n\}\subset\mathcal{P}(L^2([0,1]\to K,dt))$ be a sequence converging to $h(t,\omega)$ in $\mathbb{D}_p^{\alpha}(L^2([0,1]\to K,dt))$.
Then
\begin{align*}
&\left\| \dint_0^1h_n(t,\omega)dt-(I-L)^{-\alpha/2}\dint_0^1(I-L)^{\alpha/2}h(t,\omega)dt\right\|_{\mathbb{D}_p^{\alpha}(K)}\\
&=\left\|\dint_0^1(I-L)^{\alpha/2}(h_n(t,\omega)-h(t,\omega))dt\right\|_{L^p(K)}\\
&\leqq\|h_n(t,\omega)-h(t,\omega)\|_{\mathbb{D}_p^{\alpha}(L^2([0,1]\to K,dt))}, 
\end{align*}
which tends to $0$ as $n$ goes to $\infty$.
Note that $\int_0^1h(t,\omega)dt=(I-L)^{-\alpha/2}\int_0^1(I-L)^{\alpha/2}h(t,\omega)dt$ if $\alpha\geqq 0$, since
\begin{align*}
\left\| \dint_0^1h_n(t,\omega)dt-\dint_0^1h(t,\omega)dt\right\|_{L^p(K)}
&\leqq\|h_n(t,\omega)-h(t,\omega)\|_{L^p(L^2([0,1]\to K,dt))}\\
&\leqq\|h_n(t,\omega)-h(t,\omega)\|_{\mathbb{D}_p^{\alpha}(L^2([0,1]\to K,dt))}.
\end{align*}
\end{remark}
\section{Stochastic Fourier Coefficients and Main Result}
Let $1<p<\infty$ and $\alpha\in \mathbf{R}$.
Suppose $a(t,\omega) \in \mathbb{D}_p^{\alpha}(L^2([0,1]\to\mathbf{C}))$ and $b(t,\omega)\in \mathbb{D}_p^{\alpha-1}(L^2([0,1]\to\mathbf{C}))$.
We consider the following stochastic process $X(t,\omega)$:
\begin{equation*} \label{eq:model}
dX_t=b(t,\omega)dt+a(t,\omega)dW_t.
\end{equation*}
We set 
\begin{equation*}\label{eq:differenceX}
X(A)=\dint_0^1\mathbf{1}_{A}(u)b(u,\omega)du+\dint_0^1\mathbf{1}_{A}(u)a(u,\omega)dW_u,
\end{equation*}
where $\mathbf{1}_{A}(u)$ denotes the indicator function of a measurable set $A$. \par
Let $\psi(t)$ be a square integrable function on $[0,1]$ such that 
$\esssup_{t \in [0,1]}|\psi(t)|$ $<\infty$.
We define the integral $\int_0^1 \psi(t)dX_t$ in the following manner:
Let $\{\psi_n(t)=\sum a^n_i\mathbf{1}_{A^n_i}(t)\}$ be step functions converging to $\psi(t)$ in measure, which satisfy $\sup_n\esssup_{t \in [0,1]}|\psi_n(t)|<\infty$.
Then $\int_0^1\psi_n(t)dX_t$ is defined by $\sum a^n_iX(A^n_i)$.
We define $\int_0^1 \psi(t)dX_t$ by the limit of $\int_0^1\psi_n(t)dX_t$ if it is independent of the choice of uniformly essentially bounded step functions converging in mean to $\psi(t)$.

We see that
\begin{equation*}
\dint_0^1 \psi(t)dX_t=\dint_0^1 \psi(t)b(t,\omega)dt+\dint_0^1 \psi(t)a(t,\omega)dW_t
\end{equation*}
holds in $\mathbb{D}_p^{\alpha-1}(\mathbf{C})$, because
\begin{align*}
&\left\|\dint_0^1 \psi_n(t)a(t,\omega)dW_t-\dint_0^1 \psi(t)a(t,\omega)dW_t\right\|_{\mathbb{D}_p^{\alpha-1}(\mathbf{C})}\\
& \leqq C \bigl\| (\psi_n(t)- \psi(t))a(t,\omega)\bigr\|_{\mathbb{D}_p^{\alpha}(L^2([0,1]\to\mathbf{C}))}\\
& = C E\biggl[\Bigl( \dint_0^1\bigl|(\psi_n(t)- \psi(t))(I-L)^{\alpha/2}a(t,\omega)\bigr|^2dt\Bigr)^{p/2}\biggr]^{1/p}\\
 & = C E\biggl[\Bigl( \dint_{|\psi_n(t)- \psi(t)|\leqq\epsilon}\bigl|(\psi_n(t)- \psi(t))(I-L)^{\alpha/2}a(t,\omega)\bigr|^2dt\\
 &\phantom{= C E\biggl[\Bigl(}+\dint_{|\psi_n(t)- \psi(t)|>\epsilon}\bigl|(\psi_n(t)- \psi(t))(I-L)^{\alpha/2}a(t,\omega)\bigr|^2dt\Bigr)^{p/2}\biggr]^{1/p}\\
&\leqq C E\biggl[\Bigl(\epsilon^2 \dint_0^1\bigl|(I-L)^{\alpha/2}a(t,\omega)\bigr|^2dt\\
&+\!\bigl(\sup_n\esssup|\psi_n(t)|+\esssup|\psi(t)|\bigr)\!\!\dint_{|\psi_n(t)- \psi(t)|>\epsilon}\!\bigl|(I-L)^{\alpha/2}a(t,\omega)\bigr|^2\!dt\Bigr)^{p/2}\biggr]^{1/p}
\end{align*}
which converges to $0$ as $n\to\infty$ and $\epsilon\to 0$, and
\begin{align*}
&\left\|\dint_0^1 \psi_n(t)b(t,\omega)dt-\dint_0^1 \psi(t)b(t,\omega)dt\right\|_{\mathbb{D}_p^{\alpha-1}(\mathbf{C})}\\
&= E\biggl[\Bigl( \dint_0^1\bigl|(\psi_n(t)- \psi(t))(I-L)^{(\alpha-1)/2}b(t,\omega)\bigr|^2dt\Bigr)^{p/2}\biggr]^{1/p}
\end{align*}
converges to $0$ by the same estimation as above.

For any orthonormal 
basis $\{\varphi_n(t)\}$ in $L^2([0,1])$ such that 
 $\esssup_{t \in [0,1]}|\varphi_n(t)|$ $< \infty$ for all $n$,
we define the stochastic Fourier coefficients $\{\mathcal{F}_n(dX)\}$ of the differential $dX$ as follows;
\begin{equation*}\label{eq:stieltjes integral}
\F_n(dX)=\int_0^1 \overline{\varphi_n(t)}dX_t, \ \ (
\overline{\varphi_n(t)}= \mbox{the complex conjugate of } \varphi_n(t)).
\end{equation*}
Our main result is as follows:
\begin{theorem} \label{thm:Main2}
Let $1<p<\infty$ and $\alpha\in \mathbf{R}$.
Suppose $a(t,\omega) \in \mathbb{D}_p^{\alpha}(L^2([0,1]\to\mathbf{C}))$ and $b(t,\omega)\in \mathbb{D}_p^{\alpha-1}(L^2([0,1]\to\mathbf{C}))$.
Then 
any noncausal It\^{o} process can be identified by the complete set of the SFC 
${\mathcal F}_n(dX)$ of its differential $dX$ with respect to the complete system of 
trigonometric functions $\{e_n(t), \ n \in {\bf Z}\}$ where $e_n(t)=e^{2\pi\sqrt{-1}nt}$. 
More precisely, the Fourier coefficient (in the usual sense) %
of the coefficient $a(t,\omega)$ is determined by the following formula of Bohr product;
\begin{equation} \label{eq:Bohr2} 
\begin{array}{l}
\D{\lim_{N\rightarrow\infty}\frac{1}{2N+1}\sum_{k+\ell=n, |\ell|\leq N}\!\!
\left(\int _0^1a(t,\omega)\overline{e_k(t)}dW_t+\int_0^1b(t,\omega)\overline{e_k(t)} 
dt\right)\!\!\int_0^1\overline{e_{\ell}(t)}dW_t} \\
=\D{\int_0^1a(t,\omega)\overline{e_n(t)}dt} \quad \mathrm{in} \quad \mathbb{D}_p^{\alpha-2}(\mathbf{C}).
\end{array} 
\end{equation}
\end{theorem}
\begin{remark}
We can obtain $\{\int_0^1b(t,\omega)\overline{e_n(t)}dt, n\in\mathbf{Z}\}$ from \eqref{eq:Bohr2} above. Indeed; \\
If $\{\int_0^1a(t,\omega)\overline{e_n(t)}dt, n\in\mathbf{Z}\}$ is obtained, then $a(t,\omega)$ can be reconstructed.
Thus we get $\{\int_0^1a(t,\omega)\overline{e_n(t)}dW_t, n\in\mathbf{Z}\}$. 
Since $\int_0^1b(t,\omega)\overline{e_n(t)}dt=\mathcal{F}_n(dX)-\int_0^1a(t,\omega)\overline{e_n(t)}dW_t$, we obtain $\{\int_0^1b(t,\omega)\overline{e_n(t)}dt, n\in\mathbf{Z}\}$ .
\end{remark}
To prove the theorem above, we need multiplication formulas as follows.
\begin{proposition}\label{proposition:multiplication_p}
Let $a(t,\omega)\in\mathbb{D}_p^{\alpha}(L^2([0,1]\to\mathbf{C}))\,\,(\alpha\in\mathbf{R})$ and $e(t)\in L^2([0,1],dt)$.
Then $\int_0^1a(t,\omega)dW_t\int_0^1e(t)dW_t$ exists in $\mathbb{D}_p^{\alpha-2}(\mathbf{C})$ and the following equality holds in $\mathbb{D}_p^{\alpha-2}(\mathbf{C})$.
\allowdisplaybreaks{
\begin{align*}%
\notag&\dint_0^1a(t,\omega)dW_t\dint_0^1e(t)dW_t=\dint_0^1\left(\dint_0^1a(s,\omega)dW_s\right)e(t)dW_t\\
&\qquad\qquad\qquad\qquad+\int_0^1\left(\dint_0^1D_ta(s,\omega)e(t)dt\right)dW_s+\dint_0^1 a(t,\omega)e(t)dt.
\end{align*}
}
\end{proposition}
\begin{proposition}\label{proposition:multiplication2_p}
Let $b(t,\omega)\in\mathbb{D}_p^{\alpha}(L^2([0,1]\to\mathbf{C}))\,\,(\alpha\in\mathbf{R})$ and $e(t)\in L^2([0,1])$.
Then $\int_0^1b(t,\omega)dt\int_0^1e(t)dW_t$ exists in $\mathbb{D}_p^{\alpha-1}(\mathbf{C})$ and the following equality holds in $\mathbb{D}_p^{\alpha-1}(\mathbf{C})$.
\begin{equation*}
\dint_0^1\!b(t,\omega)dt\!\dint_0^1\!e(t)dW_t
=\dint_0^1\!\!\left(\dint_0^1\!b(s,\omega)ds\right)\!e(t)dW_t+\dint_0^1\!\!\!\dint_0^1\!D_tb(s,\omega)e(t)dtds.
\end{equation*}
%
\end{proposition}%
Proofs of Propositions \ref{proposition:multiplication_p} and  \ref{proposition:multiplication2_p} are given in the next section. %
We also remark the following estimate on trigonometric functions (see, for instance, \cite{malliavin-thalmaier}).
\begin{proposition}\label{proposition:trigonometric}
Let $e_n(t)=e^{2\pi\sqrt{-1}nt}, n\in\mathbf{Z},$ be trigonometric functions.
Then we have
\begin{equation*}
\dsum_{k+\ell=n, |\ell|\le N}\overline{e_{\ell}(s)}\overline{e_k(t)}=\overline{e_n(t)}\dfrac{\sin(N+\frac12)\pi(t-s)}{\sin\frac{1}{2}\pi(t-s)}.
\end{equation*}
Set $D_N(t)=\sin(N+\frac12)\pi t/\sin\frac{1}{2}\pi t$.
Then it holds that
\begin{equation*}
\int_0^1|D_N(s-t)|^2dt = 2N+1
\end{equation*}
for every $s\in [0,1]$.
\end{proposition}
\begin{proof}[Proof of Theorem~\ref{thm:Main2}]
From Propositions \ref{proposition:multiplication_p}, \ref{proposition:multiplication2_p} and \ref{proposition:trigonometric}, we have
\allowdisplaybreaks{
\begin{align}\notag%
&\dfrac1{2N+1}\sum_{k+\ell=n, |\ell|\le N}\left(\dint _0^1a(t,\omega)\overline{e_k(t)} dW_t+ \dint_0^1 b(t,\omega)\overline{e_k(t)} dt\right)\int_0^1\overline{e_{\ell}(t)}dW_t\\\notag
&=\int_0^1a(t,\omega)\overline{e_n(t)}dt\\\label{adWdW}%
&\quad+\dfrac1{2N+1}\int_0^1\left(\int_0^1a(s,\omega)\overline{e_n(s)}D_N(s-t) dW_s\right)dW_t\\\label{DadsdW}%
&\quad+\dfrac1{2N+1}\int_0^1\left(\int_0^1D_ta(s,\omega)\overline{e_n(s)}D_N(s-t)ds\right)dW_t\\\label{bdsdW}%
&\quad+\dfrac1{2N+1}\dint_0^1\left(\dint_0^1b(s,\omega)\overline{e_n(s)}D_N(s-t)ds\right)dW_t\\\label{Dbdsdt}%
&\quad+\dfrac1{2N+1}\dint_0^1\dint_0^1D_tb(s,\omega)\overline{e_n(s)}D_N(s-t)dsdt
\end{align}
}
in $\mathbb{D}_p^{\alpha-2}(\mathbf{C})$.
All we have to prove is that each of \eqref{adWdW}, \eqref{DadsdW}, \eqref{bdsdW} and \eqref{Dbdsdt} converges to zero if $N$ tends to infinity.
Applying \eqref{estimate:dW} and Proposition \ref{proposition:trigonometric}, we have
\allowdisplaybreaks{
\begin{align*}
&\left\|\dfrac1{2N+1}\int_0^1\left(\int_0^1a(s,\omega)\overline{e_n(s)}D_N(s-t) dW_s\right)dW_t\right\|_{\mathbb{D}_p^{\alpha-2}(\mathbf{C})}\\
&\leqq C\dfrac1{2N+1}\left\|a(s,\omega)\overline{e_n(s)}D_N(s-t)\right\|_{\mathbb{D}_p^{\alpha}(L^2([0,1],dt)\otimes L^2([0,1],ds)\otimes\mathbf{C})}\\
&=C\dfrac1{2N+1}E\left[\left(\int_0^1\int_0^1|(I-L)^{\alpha/2}a(s,\omega)\overline{e_n(s)}D_N(s-t)|^2dtds\right)^{p/2}\right]^{1/p}\\
&\leqq C\sqrt{\dfrac{1 }{2N+1}}E\left[\left(\int_0^1|(I-L)^{\alpha/2}a(s,\omega)|^2ds\right)^{p/2}\right]^{1/p}\\
&= C\sqrt{\dfrac{1 }{2N+1}}\|a(s,\omega)\|_{\mathbb{D}_p^{\alpha}(L^2([0,1]\to\mathbf{C}))}.
\end{align*}
}
Applying \eqref{estimate:dW}, Remark \ref{remark:int_hdt}, Proposition \ref{proposition:trigonometric} and \eqref{estimate:D}, we have
\allowdisplaybreaks{
\begin{align*}
&\left\|\dfrac1{2N+1}\int_0^1\left(\int_0^1D_ta(s,\omega)\overline{e_n(s)}D_N(s-t)ds\right)dW_t\right\|_{\mathbb{D}_p^{\alpha-2}(\mathbf{C})}\\
&=\dfrac1{2N+1}\times\\
&\left\|\int_0^1\!\!\left((I-L)^{-(\alpha-1)}\!\!\int_0^1\!(I-L)^{(\alpha-1)/2}D_ta(s,\omega)\overline{e_n(s)}D_N(s-t)ds\!\right)\!dW_t\right\|_{\mathbb{D}_p^{\alpha-2}(\mathbf{C})}\\
&\leqq \dfrac{C}{2N+1}E\!\left[\!\left(\!\dint_0^1\!\left|\dint_0^1(I-L)^{(\alpha-1)/2}D_ta(s,\omega)\overline{e_n(s)}D_N(s-t)ds\right|^2dt\right)^{p/2}\right]^{1/p}\\
&\leqq \dfrac{C}{2N+1}E\left[\left(\dint_0^1\left(\dint_0^1\left|(I-L)^{(\alpha-1)/2}D_ta(s,\omega)\right|^2ds\right)\times\right.\right.\\
&\qquad\qquad\qquad\qquad\qquad\qquad\qquad\qquad\qquad\left.\left.\left(\dint_0^1|D_N(s-t)|^2ds\right)dt\right)^{p/2}\right]^{1/p}\\
&\leqq C\sqrt{\dfrac{1 }{2N+1}}E\left[\left(\dint_0^1\dint_0^1\left|(I-L)^{(\alpha-1)/2}D_ta(s,\omega)\right|^2dtds\right)^{p/2}\right]^{1/p}\\
&= C\sqrt{\dfrac{1 }{2N+1}}\left\|Da(s,\omega)\right\|_{\mathbb{D}_{p}^{\alpha-1}(L^2([0,1],dt)\otimes L^2([0,1];\mathbf{C}))}\\
&\leqq C\sqrt{\dfrac{1 }{2N+1}}\left\|a(s,\omega)\right\|_{\mathbb{D}_{p}^{\alpha}(L^2([0,1];\mathbf{C}))}.
\end{align*}
}
Applying \eqref{estimate:dW}, Remark \ref{remark:int_hdt} and Proposition \ref{proposition:trigonometric}, we have
\allowdisplaybreaks{
\begin{align*}
&\left\|\dfrac1{2N+1}\dint_0^1\left(\dint_0^1b(s,\omega)\overline{e_n(s)}D_N(s-t)ds\right)dW_t\right\|_{\mathbb{D}_p^{\alpha-2}(\mathbf{C})}\\
&\leqq C\dfrac1{2N+1} \left\|\int_0^1b(s,\omega)\overline{e_n(s)}D_N(s-t)ds\right\|_{\mathbb{D}_p^{\alpha-1}(L^2([0,1]:\mathbf{C}))}\\
&=\dfrac{C}{2N+1} E\left[\left(\dint_0^1\left|(I-L)^{(\alpha-1)/2}\int_0^1b(s,\omega)\overline{e_n(s)}D_N(s-t)ds\right|^2dt\right)^{p/2}\right]^{1/p}\\
&=\dfrac{C}{2N+1} E\left[\left(\dint_0^1\left|\int_0^1(I-L)^{(\alpha-1)/2}b(s,\omega)\overline{e_n(s)}D_N(s-t)ds\right|^2dt\right)^{p/2}\right]^{1/p}\\
&\leqq\dfrac{C}{2N+1} E\left[\left(\dint_0^1\int_0^1\left|(I-L)^{(\alpha-1)/2}b(s,\omega)\overline{e_n(s)}D_N(s-t)\right|^2dtds\right)^{p/2}\right]^{1/p}\\%
&\leqq C\sqrt{\dfrac{1 }{2N+1}} E\left[\left(\dint_0^1\left|(I-L)^{(\alpha-1)/2}b(s,\omega)\right|^2ds\right)^{p/2}\right]^{1/p}\\
&=C\sqrt{\dfrac{1 }{2N+1}} \|b(s,\omega)\|_{\mathbb{D}_p^{\alpha-1}(L^2([0,1]:\mathbf{C}))}.
\end{align*}
}
Applying Remark \ref{remark:int_hdt}, Proposition \ref{proposition:trigonometric} and \eqref{estimate:D}, we have
\allowdisplaybreaks{
\begin{align*}
&\left\|\dfrac1{2N+1}\dint_0^1\dint_0^1D_tb(s,\omega)\overline{e_n(s)}D_N(s-t)dsdt\right\|_{\mathbb{D}_p^{\alpha-2}(\mathbf{C})}\\
&=\dfrac1{2N+1}\times\\
&\quad\left\|(I-L)^{-(\alpha-2)/2}\!\dint_0^1\!\!\dint_0^1(I-L)^{(\alpha-2)/2}D_tb(s,\omega)\overline{e_n(s)}D_N(s-t)dsdt\right\|_{\mathbb{D}_p^{\alpha-2}(\mathbf{C})}\\
&=\dfrac1{2N+1}E\left[\left|\dint_0^1\dint_0^1(I-L)^{(\alpha-2)/2}D_tb(s,\omega)\overline{e_n(s)}D_N(s-t)dsdt\right|^p\right]^{1/p}\\
&\leqq\dfrac1{2N+1}E\left[\left(\dint_0^1\dint_0^1\left|(I-L)^{(\alpha-2)/2}D_tb(s,\omega)\right|^2dsdt\right)^{p/2}\times\right.\\
&\qquad\qquad\qquad\qquad\qquad\qquad\qquad\qquad\left.\left(\dint_0^1\dint_0^1|D_N(s-t)|^2dsdt\right)^{p/2}\right]^{1/p}\\
&\leqq\sqrt{\dfrac{1 }{2N+1}}E\left[\left(\dint_0^1\dint_0^1\left|(I-L)^{(\alpha-2)/2}D_tb(s,\omega)\right|^2dsdt\right)^{p/2}\right]^{1/p}\\
&=\sqrt{\dfrac{1 }{2N+1}}\left\|Db(s,\omega)\right\|_{\mathbb{D}_p^{\alpha-2}(L^2([0,1],dt)\otimes L^2([0,1]:\mathbf{C}))}\\
&\leqq\sqrt{\dfrac{1 }{2N+1}}\left\|b(s,\omega)\right\|_{\mathbb{D}_p^{\alpha-1}(L^2([0,1]:\mathbf{C}))}.
\end{align*}
}
Therefore each of \eqref{adWdW}, \eqref{DadsdW}, \eqref{bdsdW} and \eqref{Dbdsdt} converges to zero if $N$ tends to infinity, which completes the proof.
\end{proof}
\section{Proofs of Propositions}
To prove Propositions \ref{proposition:multiplication_p} and \ref{proposition:multiplication2_p}, we prepare the next lemma.
\begin{lemma}\label{lemma:multiplicaton_Fdelta}
Let $F(\omega) \in\mathbb{D}_p^{\alpha}(\mathbf{C})\,\, (\alpha\in\mathbf{R})$ and $e(t)\in L^2([0,1],dt)$.
Then $F(\omega)\int_0^1e(t)dW_t$ exists in $\mathbb{D}_p^{\alpha-1}(\mathbf{C})$ and the following equality holds in $\mathbb{D}_p^{\alpha-1}(\mathbf{C})$.
\begin{equation}\label{Fdelta}
 F(\omega)\int_0^1e(t)dW_t=\int_0^1F(\omega)e(t)dW_t+\dint_0^1D_tF(\omega)e(t)dt.
 \end{equation}
\end{lemma}
\begin{proof}%
Let $\{F_n(\omega)\}\subset\mathcal{P}(\mathbf{C})$ be a sequence converging to $F(\omega)$ in $\mathbb{D}_p^{\alpha}(\mathbf{C})$.
Then
\begin{equation*}
F_n(\omega)\dint_0^1e(t)dW_t=\dint_0^1F_n(\omega)e(t)dW_t+\dint_0^1D_tF_n(\omega)e(t)dt
\end{equation*}
holds. 
Since $F_n(\omega)$ converges to $F(\omega)$ in $\mathbb{D}_p^{\alpha}(\mathbf{C})$, $F_n(\omega)\int_0^1e(t)dW_t$ converges to $F(\omega)\int_0^1e(t)dW_t$  in $\mathbb{D}_r^{\alpha}(\mathbf{C})$ for any $r$ $(r<p)$.
Moreover, $\int_0^1F_n(\omega)e(t)dW_t$ converges to $\int_0^1F(\omega)e(t)dW_t$ in $\mathbb{D}_p^{\alpha-1}(\mathbf{C})$, and $D_tF_n(\omega)$ to $D_tF(\omega)$ in $\mathbb{D}_p^{\alpha-1}($ $L^2([0,1]:\mathbf{C}))$.
Therefore 
$$\dlim_{n\to\infty}\dint_0^1D_tF_n(\omega)e(t)dt=\dint_0^1D_tF(\omega)e(t)dt$$
in $\mathbb{D}_p^{\alpha-1}(\mathbf{C})$.
Since both $\mathbb{D}_r^{\alpha}(\mathbf{C})$ and $\mathbb{D}_p^{\alpha-1}(\mathbf{C})$ are subsets of $\mathbb{D}_r^{\alpha-1}(\mathbf{C})$, \eqref{Fdelta} holds in $\mathbb{D}_r^{\alpha-1}(\mathbf{C})$, while the right hand side of \eqref{Fdelta} belongs to $\mathbb{D}_p^{\alpha-1}(\mathbf{C})$.
Therefore \eqref{Fdelta} holds in $\mathbb{D}_p^{\alpha-1}(\mathbf{C})$.
\end{proof}
\begin{proof}[Proof of Proposition~\ref{proposition:multiplication_p}]
Since $D_t\int_0^1a(s,\omega)dW_s=\int_0^1D_ta(s,\omega)dW_s+a(t,\omega)$, we have
\allowdisplaybreaks{
\begin{align*}
&\dint_0^1a(s,\omega)dW_s\dint_0^1e(t)dW_t\\
&=\dint_0^1\left(\dint_0^1a(s,\omega)dW_s\right)e(t)dW_t
+\dint_0^1\left(\dint_0^1D_ta(s,\omega)dW_s\right)e(t)dt\\
&\quad+\dint_0^1a(t,\omega)e(t)dt
\end{align*}
}
in $\mathbb{D}_p^{\alpha-2}(\mathbf{C})$ from Lemma \ref{lemma:multiplicaton_Fdelta}. 
Thus all we have to do is to show that the following Fubini-type formula holds in  $\mathbb{D}_p^{\alpha-2}(\mathbf{C})$:
\allowdisplaybreaks{
\begin{equation*}
\dint_0^1\left(\dint_0^1D_ta(s,\omega)dW_s\right)e(t)dt
=\dint_0^1\left(\dint_0^1D_ta(s,\omega)e(t)dt\right)dW_s.
\end{equation*}
}
Let $G(\omega)\in\mathbb{D}_q^{-(\alpha-2)}(\mathbf{C})$.
Then
\allowdisplaybreaks{
\begin{align*}
&{}_{{}_{\mathbb{D}_p^{\alpha-2}(\mathbf{C})}^{}}^{}\left\langle \dint_0^1\left(\dint_0^1D_ta(s,\omega)dW_s\right)e(t)dt, G(\omega)\right\rangle{}_{_{\mathbb{D}_q^{-(\alpha-2)}(\mathbf{C})}^{}}^{}\\
&={}_{{}_{\mathbb{D}_p^{\alpha-2}(\mathbf{C})}^{}}^{}\left\langle (I-L)^{-(\alpha-2)/2}\dint_0^1(I-L)^{(\alpha-2)/2}\left(\dint_0^1D_ta(s,\omega)dW_s\right)e(t)dt,\right.\\
&\left.\phantom{ (I-L)^{-(\alpha-2)/2}\dint_0^1(I-L)^{(\alpha-2)/2}\left(\dint_0^1D_ta(s,\omega)dW_s\right)e(t)dt} G(\omega)\right\rangle{}_{_{\mathbb{D}_q^{-(\alpha-2)}(\mathbf{C})}^{}}^{}\\
&=E\left[\left(\dint_0^1(I-L)^{(\alpha-2)/2}\left(\dint_0^1D_ta(s,\omega)dW_s\right)e(t)dt\right)\overline{(I-L)^{-(\alpha-2)/2}G(\omega)}\right]\\
&=\dint_0^1E\left[\left((I-L)^{(\alpha-2)/2}\left(\dint_0^1D_ta(s,\omega)dW_s\right)e(t)\right)\overline{(I-L)^{-(\alpha-2)/2}G(\omega)}\right]dt\\
&=\dint_0^1{}_{{}_{\mathbb{D}_p^{\alpha-2}(\mathbf{C})}^{}}^{}\left\langle\left(\dint_0^1D_ta(s,\omega)dW_s\right)e(t),G(\omega)\right\rangle{}_{{}_{\mathbb{D}_q^{-(\alpha-2)}(\mathbf{C})}^{}}^{}dt\\
&=\dint_0^1\dint_0^1{}_{\mathbb{D}_p^{\alpha-1}(\mathbf{C})}\left\langle D_ta(s,\omega)e(t),D_sG(\omega)\right\rangle_{\mathbb{D}_q^{-(\alpha-1)}(\mathbf{C})}dsdt\\
&=\dint_0^1\dint_0^1E\left[\left((I-L)^{(\alpha-1)/2}D_ta(s,\omega)e(t)\right)\overline{(I-L)^{-(\alpha-1)/2}D_sG(\omega)}\right]dsdt\\
&=\dint_0^1E\left[\left(\dint_0^1(I-L)^{(\alpha-1)/2}D_ta(s,\omega)e(t)dt\right)\overline{(I-L)^{-(\alpha-1)/2}D_sG(\omega)}\right]ds\\
&=\dint_0^1{}_{{}_{\mathbb{D}_p^{\alpha-1}(\mathbf{C})}^{}}^{}\left\langle (I-L)^{-(\alpha-1)/2}\dint_0^1(I-L)^{(\alpha-1)/2}D_ta(s,\omega)e(t)dt, \right.\\
&\left.\phantom{(I-L)^{-(\alpha-1)/2}\dint_0^1(I-L)^{(\alpha-1)/2}D_ta(s,\omega)e(t)dt,}D_sG(\omega)\right\rangle{}_{{}_{\mathbb{D}_q^{-(\alpha-1)}(\mathbf{C})}^{}}^{}ds\\
&={}_{{}_{\mathbb{D}_p^{\alpha-2}(\mathbf{C})}^{}}^{}\left\langle\dint_0^1(I-L)^{-(\alpha-1)/2}\left(\dint_0^1(I-L)^{(\alpha-1)/2}D_ta(s,\omega)e(t)dt\right)dW_s,\right.\\
&\left.\phantom{\dint_0^1(I-L)^{-(\alpha-1)/2}\left(\dint_0^1(I-L)^{(\alpha-1)/2}D_ta(s,\omega)e(t)dt\right)dW_s} G(\omega)\right\rangle{}_{{}_{\mathbb{D}_q^{-(\alpha-2)}(\mathbf{C})}^{}}^{}\\
&={}_{{}_{\mathbb{D}_p^{\alpha-2}(\mathbf{C})}^{}}^{}\left\langle\dint_0^1\left(\dint_0^1D_ta(s,\omega)e(t)dt\right)dW_s, G(\omega)\right\rangle{}_{{}_{\mathbb{D}_q^{-(\alpha-2)}(\mathbf{C})}^{}}^{}
\end{align*}
}
which is the desired result.
\end{proof}
\begin{proof}[Proof of Proposition~\ref{proposition:multiplication2_p}]
Substituting $\int_0^1b(s,\omega)ds$ to $F(\omega)$ in \eqref{Fdelta}, we have
\allowdisplaybreaks{
\begin{align*}
&\dint_0^1b(t,\omega)dt\dint_0^1e(t)dW_t\\
&=\dint_0^1\left(\dint_0^1b(s,\omega)ds\right)e(t)dW_t+\dint_0^1D_t\left(\dint_0^1b(s,\omega)ds\right)e(t)dt.
\end{align*}
}
Therefore it is enough to show that
\begin{equation*}
\dint_0^1D_t\left(\dint_0^1b(s,\omega)ds\right)e(t)dt=\dint_0^1\dint_0^1D_tb(s,\omega)e(t)dtds
\end{equation*}
holds in $\mathbb{D}_p^{\alpha-1}(\mathbf{C})$.
Let $G(\omega)\in \mathbb{D}_q^{-(\alpha-1)}(\mathbf{C})$.
Then we have
\allowdisplaybreaks{
\begin{align*}
&{}_{{}_{\mathbb{D}_p^{\alpha-1}(\mathbf{C})}^{}}^{}\left\langle \dint_0^1D_t\left(\dint_0^1b(s,\omega)ds\right)e(t)dt, G(\omega)\right\rangle{}_{{}_{\mathbb{D}_q^{-(\alpha-1)}(\mathbf{C})}^{}}^{}\\
&=E\left[\left(\dint_0^1(I-L)^{(\alpha-1)/2}D_t\left(\dint_0^1b(s,\omega)ds\right)e(t)dt\right)\overline{(I-L)^{-(\alpha-1)/2}G(\omega)}\right]\\
&=\dint_0^1E\left[(I-L)^{(\alpha-1)/2}D_t\left(\dint_0^1b(s,\omega)ds\right)e(t)\overline{(I-L)^{-(\alpha-1)/2}G(\omega)}\right]dt\\
&=\dint_0^1{}_{{}_{\mathbb{D}_p^{\alpha-1}(\mathbf{C})}^{}}^{}\left\langle D_t\dint_0^1b(s,\omega)ds,G(\omega)\overline{e(t)}\right\rangle{}_{{}_{\mathbb{D}_q^{-(\alpha-1)}(\mathbf{C})}^{}}^{}dt\\
&={}_{{}_{\mathbb{D}_p^{\alpha}(\mathbf{C})}^{}}^{}\left\langle \dint_0^1b(s,\omega)ds,\dint_0^1G(\omega)\overline{e(t)}dW_t\right\rangle{}_{{}_{\mathbb{D}_q^{-\alpha}(\mathbf{C})}^{}}^{}\\
&=E\left[\left((I-L)^{\alpha/2}\dint_0^1b(s,\omega)ds\right)\overline{(I-L)^{-\alpha/2}\dint_0^1G(\omega)\overline{e(t)}dW_t}\right]\\
&=E\left[\left(\dint_0^1(I-L)^{\alpha/2}b(s,\omega)ds\right)\overline{(I-L)^{-\alpha/2}\dint_0^1G(\omega)\overline{e(t)}dW_t}\right]\\
&=\dint_0^1E\left[(I-L)^{\alpha/2}b(s,\omega)\overline{(I-L)^{-\alpha/2}\dint_0^1G(\omega)\overline{e(t)}dW_t}\right]ds\\
&=\dint_0^1{}_{{}_{\mathbb{D}_p^{\alpha}(\mathbf{C})}^{}}^{}\left\langle b(s,\omega),\dint_0^1G(\omega)\overline{e(t)}dW_t\right\rangle{}_{{}_{\mathbb{D}_q^{-\alpha}(\mathbf{C})}^{}}^{}ds\\
&=\dint_0^1\dint_0^1{}_{{}_{\mathbb{D}_p^{\alpha-1}(\mathbf{C})}^{}}^{}\left\langle D_tb(s,\omega),G(\omega)\overline{e(t)}\right\rangle{}_{{}_{\mathbb{D}_q^{-(\alpha-1)}(\mathbf{C})}^{}}^{}dtds\\
&=\dint_0^1\dint_0^1E\left[(I-L)^{(\alpha-1)/2}D_tb(s,\omega)\overline{(I-L)^{-(\alpha-1)/2}G(\omega)\overline{e(t)}}\right]dtds\\
&=E\left[\dint_0^1\dint_0^1(I-L)^{(\alpha-1)/2}D_tb(s,\omega)e(t)dtds\overline{(I-L)^{-(\alpha-1)/2}G(\omega)}\right]\\
&={}_{{}_{\mathbb{D}_p^{\alpha-1}(\mathbf{C})}^{}}^{}\left\langle \dint_0^1\dint_0^1D_tb(s,\omega)e(t)dtds, G(\omega)\right\rangle{}_{{}_{\mathbb{D}_q^{-(\alpha-1)}(\mathbf{C})}^{}}^{},
\end{align*}
}
which completes the proof.
\end{proof}
\section*{Acknowledgment}
This work was partially supported by JSPS KAKENHI Grant Number 25400135.
%
%
%

\end{document}